\documentclass[12pt,a4paper]{article}
\usepackage{amsmath,amssymb}
\usepackage{enumerate,a4}
\usepackage[english]{babel}
\usepackage[utf8]{inputenc}
\usepackage{url}
\usepackage{amsthm}
\usepackage{amsmath,amssymb}
\usepackage{multirow}
\usepackage[table,xcdraw]{xcolor}
\usepackage{array}
\usepackage{booktabs}
\usepackage{tabularx}
\usepackage{todonotes}
\usepackage{caption}
\usepackage{subcaption}
\usepackage{todonotes}
\usepackage{float}
\usepackage[margin=1in]{geometry}

\usepackage[font=footnotesize,labelfont=bf]{caption}

\usepackage[ruled,vlined]{algorithm2e}
\SetKwRepeat{Do}{do}{while}

\usepackage[utf8]{inputenc}

\newcommand{\B}{{\mathcal{B}}}
\newcommand{\TRI}{{\scriptscriptstyle {\textrm{TRI}}}}
\newcommand{\PENT}{{\scriptscriptstyle {\textrm{PENT}}}}
\newcommand{\HYP}{{\scriptscriptstyle {\textrm{HYP}}}}
\newcommand{\SDP}{{\scriptscriptstyle {\textrm{SDP}}}}
\newcommand{\BASIC}{{\scriptscriptstyle {\textrm{BASIC}}}}

\newcommand{\PreserveBackslash}[1]{\let\temp=\\#1\let\\=\temp}
\newcolumntype{C}[1]{>{\PreserveBackslash\centering}p{#1}}
\newcolumntype{R}[1]{>{\PreserveBackslash\raggedleft}p{#1}}
\newcolumntype{L}[1]{>{\PreserveBackslash\raggedright}p{#1}}

\def\nd{\noindent}

\usepackage[T1]{fontenc}

\DeclareMathOperator{\diag}{diag}
\DeclareMathOperator{\Diag}{Diag}
\DeclareMathOperator{\rank}{rank}
\DeclareMathOperator{\diff}{\mathit{diff}}

\DeclareMathOperator*{\argmin}{arg\,min}

\newtheorem{proposition}{Proposition}
\newtheorem{theorem}{Theorem}

\DeclareFontFamily{OT1}{manual}{}
\DeclareFontShape{OT1}{manual}{m}{n}{ <10> manfnt }{}

\usepackage{titlesec}

\titleformat{\section}
  {\normalfont\fontsize{12}{15}\bfseries}{\thesection}{1em}{}
  
  \titleformat{\subsection}
  {\normalfont\fontsize{10}{15}\bfseries}{\thesubsection}{1em}{}

\begin{document}

\setlength{\abovedisplayskip}{6pt}
\setlength{\belowdisplayskip}{6pt}

	\begin{center}{\normalsize\bf
	{\large{\texttt{MADAM}}}: A PARALLEL EXACT SOLVER FOR MAX-CUT BASED ON SEMIDEFINITE PROGRAMMING AND ADMM
		}\\[16pt]
		{\bf Timotej Hrga$^{*}$, Janez Povh$^{*,\dag}$}\\[3mm]
		{\small
			 $^{*}$ University of Ljubljana, Faculty of Mechanical Engineering, Slovenia \\
			$^{\dag}$Institute of Mathematics, Physics and Mechanics Ljubljana, Slovenia \\
			{\tt email:} {\tt timotej.hrga@lecad.fs.uni-lj.si, janez.povh@lecad.fs.uni-lj.si} }\\[18pt]
			
	\end{center}

\small{ \noindent \textbf{Abstract}\\ \nd 
We present \texttt{MADAM}, a parallel semidefinite-based exact solver for Max-Cut, a problem of finding the cut with the maximum weight in a given graph. The algorithm uses the branch and bound paradigm that applies the alternating direction method of multipliers as the bounding routine to solve the basic semidefinite relaxation strengthened by a subset of hypermetric inequalities. The benefit of the new approach is a less computationally expensive update rule for the dual variable with respect to the inequality constraints. We provide a theoretical convergence of the algorithm as well as extensive computational experiments with this method, to show that our algorithm outperforms state-of-the-art approaches. Furthermore, by combining algorithmic ingredients from the serial algorithm, we develop an efficient distributed parallel solver based on MPI. 
	}\\[1mm]
	{\small\nd\textbf{Keywords:} semidefinite programming, alternating direction method of multipliers, maximum cut problem, parallel computing}

\begin{section}{Introduction}

\begin{subsection}{Motivation}
The Max-Cut problem is a classical NP-hard optimization problem \cite{garey2002computers, karp1972reducibility} on graphs with the quadratic
objective function and unconstrained binary variables. During the last decades, it has attracted the interest of many researchers, from its
theoretical and algorithmic perspective to its applicability in different fields, for instance mathematics, physics, and computer science 
\cite{krislock2014improved, liers2004computing, rendl2010solving}. Several exact algorithmic approaches have been proposed in the literature, including the BiqMac \cite{rendl2010solving} and BiqCrunch solvers \cite{krislock2017biqcrunch}, which are among the best solvers and are based on semidefinite programming.

Lassere \cite{lasserre2016max} has proved that the Max-Cut problem can be considered as a canonical model of linearly constrained linear and quadratic $0/1$ programs. The transformation is based on the exact penalty approach and was further explored and advanced by Gusmeroli and Wiegele \cite{gusmeroli2019expedis}. 
Even for instances of moderate size, it is considered a computational challenge to solve the Max-Cut to optimality. In practice, we typically solve such problems only approximately by using a heuristic or an approximation algorithm \cite{dunning2018works, goemans1995improved}. However, to compare these algorithms and evaluate their performance, we still require the optimum solutions. 
Considering all this, solving increasingly large instances of Max-Cut to optimality on parallel computers is highly needed in scientific computing. 
\end{subsection}

\begin{subsection}{Problem formulation and notations} 
The central problem that we consider is the Max-Cut problem, which can be defined as follows. For a given undirected graph $G=(V,E)$ on $n = \vert V \vert$ vertices and with edge weights $w_e$ for $e \in E$,  the Max-Cut problem asks to find a bipartition of the vertices such that the sum of the weights of the edges across the bipartition is maximum.  Let $A = (a_{ij})$ denote the weighted adjacency matrix with $a_{ij} = a_{ji} = w_e$ for edge $e = \{i, j\} \in E$ and zero otherwise. 
Encoding the partitions by vectors $z = \{0,1\}^n$, we obtain the following unconstrained binary quadratic optimization problem formulation for Max-Cut:
\begin{equation}
\textrm{maximize } \ \sum_{i,j}^n a_{ij}z_i(1-z_j) = \ z^TL_0z \ \  \textrm{ subject to }\  z \in \{0,1\}^n, \label{maxcut}
\end{equation}
where $L_0$ is the Laplacian matrix of the graph defined by $L_0 = \textrm{ diag}(Ae) - A$. By $e$ we denote the vector of all ones. 
Note that due to the symmetry of the problem, we can fix the last element of vector $z$ to zero and thus remove the last row and column of $L_0$ to obtain the matrix $\widehat{L_0} \in \mathbb{R}^{(n-1) \times (n-1)}$. 

The inner product on the space of symmetric matrices is given by $X \bullet Y = \langle X, Y \rangle = \textrm{tr}(XY) = \sum_{i,j} X_{ij}Y_{ij}$ and  the associated Frobenious norm is defined by $\Vert X \Vert_F = \sqrt{\langle X, X \rangle} = \sqrt{\sum_{i,j}X_{ij}^2}$. 
Using the property of inner product $z^TL_0z = \langle L_0, zz^T \rangle$, we can reformulate problem \eqref{maxcut} as:
\begin{align}
\label{maxcut_matrix}
\begin{split}
\textrm{maximize } \quad&
\begin{bmatrix}
\widehat{L_0} & 0 \\
0 & 0
\end{bmatrix}
\bullet
\begin{bmatrix}
zz^T & z \\
z^T & 1
\end{bmatrix} \\[0.5em]
\textrm{subject to} \quad& z \in \{ 0,1\}^{n-1}.
\end{split}
\end{align}
We increase the dimension of the problem by one, since the last column of the solution matrix of a semidefinite relaxation will be used for determining the next branching variable. 

In the following, we denote the diagonal matrix, which has $v$ on its diagonal, by $\Diag (v)$, and the vector obtained by extracting the main diagonal from the matrix $X$ is denoted by $\diag(X)$. For the given symmetric matrices $A_i$, $i = 1,\ldots,m$, let $\mathcal{A}\colon \mathcal{S}_n \to \mathbb{R}^m$ denote the linear operator mapping $n \times n$ symmetric matrices to $\mathbb{R}^m$ with $\mathcal{A}(X)_i = \langle A_i, X \rangle$. Its adjoint is well known to be $\mathcal{A}^T(y) = \sum_iy_iA_i$. 
For some real number $a$, we denote its nonnegative part by $a_+ = \max\{a, 0\}$. This definition is extended to vectors as follows: if $x \in \mathbb{R}^n$, then $\left( x_+\right)_i = \left( x_i\right)_+$, for $i = 1,\ldots, n$. We denote the projection of some symmetric matrix $X$ onto the positive semidefinite cone by $X_+$ and its projection onto the negative semidefinite cone by $X_-$. More specifically, if the eigendecomposition of $X$ is given by $X = S\Diag(\lambda)S^T$ with the eigenvalues $\lambda \in \mathbb{R}^n$ and orthogonal matrix $S \in \mathbb{R}^{n \times n}$, then we have
\begin{equation*}
X_+ = S\Diag(\lambda_+)S^T, \quad X_- = S\Diag(\lambda_-)S^T.
\end{equation*} 

\end{subsection}

\begin{subsection}{Related work and our contribution}
Over the decades, many methods have been proposed for finding exact solutions of Max-Cut. Some of them are linear programming-based methods \cite{ de1994cutting, liers2004computing}, which work particularly well when the underlying graph is sparse. Other algorithms combine semidefinite programming with the polyhedral approach \cite{helmberg1998solving} to strengthen the basic SDP relaxation with cutting planes. 
Two of them, the BiqMac \cite{rendl2010solving} and BiqCrunch \cite{krislock2017biqcrunch} solvers, respectively, have turned out to be the best performing solvers for Max-Cut in the last decade, also when compared to commercial solvers. Both solvers utilize the branch and bound (B\&B) paradigm. However, the distinction is in using different algorithms to solve the underlying SDP relaxation. Furthermore, they do not use any parallelization.
For other recent computational approaches, the reader is referred to \cite{palagi2012computational}. 

Recently, we have developed a BiqBin solver \cite{BiqBin:2020} with other authors for the class of  binary quadratic problems with linear constraints, which includes the Max-Cut problem. This solver uses the exact penalty approach to reformulate every instance of the binary quadratic problem into an instance of the Max-Cut problem and then solves it using an enhanced version of BiqMac where the underlying SDP relaxations are tight due to the inclusion of hypermetric inequalities. Additionally, the B\&B part of BiqBin has been improved compared to BiqMac, and the solver has been fully parallelized. Extensive numerical evidence shows that BiqBin outperforms BiqMac, BiqCrunch, GUROBI, and SCIP on Max-Cut instances and also on some classes of binary quadratic problems.

The main contribution of this paper is the introduction of the parallel exact solver {\texttt{MADAM}} for Max-Cut, based on the alternating direction method of multipliers (ADMM), which is  introduced in \cite{wen2010alternating}.
More precisely, we:

\begin{itemize}
	\item Adapt ADMM to solve SDP relaxations of Max-Cut, with a subset of hypermetric inequalities included, and provide a convergence proof. More specifically, we strengthen the basic SDP relaxation with triangle, pentagonal, and heptagonal inequalities. The key idea in solving the underlying relaxation is to introduce a slack variable in order to eliminate solving a quadratic program over the nonnegative orthant when computing the dual variable corresponding to inequality constraints. Instead, a sparse system of linear equations is solved and a projection onto the nonnegative orthant is performed.
	\item We propose a rounding procedure which rounds the (nearly feasible) solution of the dual problem obtained by ADMM to feasible solutions and thus provides a valid upper bound for the optimum value of Max-Cut.
	\item We develop a new B\&B algorithm for Max-Cut 
 which exploits the valid upper bounds derived by rounding the nearly feasible solution of ADMM.	
	The B\&B is parallelized based on the load coordinator-worker paradigm using the Message Passing Interface (MPI) library. Scalability results for {\tt MADAM} show good scaling properties when running on a supercomputer. 
	\item We develop a C implementation of all modules and name it {\tt MADAM} -- \texttt{M}ax-Cut \texttt{A}lternating \texttt{D}irection \texttt{A}ugmented Lagrangian \texttt{M}ethod. The source code  is available at 
	$$
	\mathrm{\emph{https://github.com/HrgaT/MADAM}.}
	$$
	{\tt MADAM} is tested against BiqMac, BiqCrunch, and BiqBin on several Max-Cut instances from the BiqMac library. Numerical results show that {\tt MADAM} outperforms other approaches. The reason is that our adaption of ADMM has a better balance between the time needed to solve the SDP relaxation and the quality of the solution than other solvers.

\end{itemize}

The main difference between {\tt MADAM}, BiqMac, and BiqBin is that {\tt MADAM} applies ADMM to solve the underlying SDP relaxations, while BiqMac and BiqBin use the bundle method.
BiqMac strengthens the basic SDP relaxation with triangle inequalities, whereas 
BiqBin and {\tt MADAM} also include a subset of pentagonal and heptagonal inequalities, which make these relaxations harder to solve.
Using ADMM in {\tt MADAM} implies that in each B\&B node, we have a (nearly) feasible solution for the SDP relaxation available, which is not the case with the bundle method applied to the partial Lagrangian dual.  
This gives us more information on how to generate feasible solutions and how to branch. 
Furthermore, the obtained SDP bounds are tighter,
and consequently the size of the explored B\&B tree is smaller compared to other solvers. 
Parallelization of B\&B  in {\tt MADAM} is very similar to the parallelization of B\&B  in BiqBin, since both are based on the load coordinator-workers paradigm and use the
MPI library for communication between the processes. 

The rest of this paper is organized as follows. Section 2 deals with semidefinite relaxations of the Max-Cut used in this paper. 
We also describe the class of cutting planes that we use. 
The description of other solution approaches that are based on semidefinite programming is given in Section 3. 
 In Section 4, we outline our new ADMM algorithm, discuss implementation details, and provide a proof of convergence. Details on how the proposed ADMM method is used within the B\&B framework are given in Section 5, while Section 6 presents numerical results of the serial algorithm. In Section 7, we give insights on how parallelization improves the performance of the {\texttt{MADAM}} solver. We conclude
with a computational experience of our method and present scalability results. 

\end{subsection}

\end{section}

\begin{section}{Semidefinite relaxation of Max-Cut}
The semidefinite program (SDP) optimizes a linear objective function of the matrix over the intersection of the cone of positive semidefinite matrices with an affine space \cite{wolkowicz2012handbook}. Compared to linear programming, SDP relaxations have been shown to provide tighter bounds on the optimum value for many combinatorial optimization problems. Particularly for the Max-Cut problem, Goemans and Williamson \cite{goemans1995improved} achieve the celebrated $0.879$ approximation ratio using the SDP relaxation. 

Since we will consider the semidefinite relaxation of \eqref{maxcut_matrix}, it is more appropriate to work with $-1/1$ variables. We  apply the change of variables between $z \in \{ 0,1\}^{n-1}$ and $x \in \{ -1,1\}^{n-1}$, defined by $z = \frac{1}{2}(x+e)$. This can be written in matrix form as 
\begin{equation*}
\begin{bmatrix}
z\\[0.2em]
1
\end{bmatrix}
=
\begin{bmatrix}
\frac{1}{2}I & \frac{1}{2}e\\[0.2em]
0 & 1
\end{bmatrix}
\begin{bmatrix}
x\\[0.2em]
1
\end{bmatrix}
= U
\begin{bmatrix}
x\\[0.2em]
1
\end{bmatrix}
\end{equation*}
and therefore
\begin{equation*}
\begin{bmatrix}
zz^T & z \\[0.2em]
z^T & 1
\end{bmatrix}
= U
\begin{bmatrix} xx^T & x \\[0.2em] x^T & 1 \end{bmatrix}
U^T.
\end{equation*}
Observe that for any $x = \{-1,1\}^{n-1}$, the matrix $xx ^T$ is positive semidefinite and its diagonal is equal to the vector of all ones. Using the above transformation, this allows us to write the Max-Cut problem in $-1/1$ variables as
\begin{align*}
\textrm{maximize } \quad& \langle L, X \rangle \\
\textrm{subject to} \quad&\diag (X) = e \\
\quad& X = \begin{bmatrix} xx^T & x \\[0.2em] x^T & 1 \end{bmatrix}
\end{align*}
or equivalently
\begin{align}\label{MC} 
\tag{MC}
\begin{split}
\textrm{maximize } \quad& \langle L, X \rangle \\
\textrm{subject to} \quad&\diag (X) = e \\
\quad& X \succeq 0\\
\quad& \rank(X) = 1,
\end{split}
\end{align}
where
\begin{equation*}
L = 
U^T \begin{bmatrix}
\widehat{L_0} & 0 \\[0.2em]
0 & 0
\end{bmatrix} U =
\frac{1}{4}
\begin{bmatrix} \widehat{L_0} &\widehat{L_0}e \\[0.2em] e^T\widehat{L_0} & e^T \widehat{L_0}e \end{bmatrix} \in \mathbb{R}^{n \times n}.
\end{equation*}
By dropping the rank-one constraint, we obtain the basic SDP relaxation
\begin{align}\label{basic_SDP} 
\tag{$\textrm{MC}_{\BASIC}$}
\begin{split}
\textrm{maximize } \quad& \langle L, X \rangle \\
\textrm{subject to} \quad&\diag (X) = e \\
\quad& X \succeq 0.
\end{split}
\end{align}

Interior-point methods (IPM) turned out to be the most prominent algorithms for solving SDPs. There exist several numerical packages to solve semidefinite programs, \emph{e.g.}, CSDP \cite{borchers1999csdp}, MOSEK \cite{mosek2015mosek}, SeDuMi \cite{sturm1999using}, and SDPT3 \cite{toh1999sdpt3}. Many of them are based on some variant of the primal-dual path following interior-point method, see for example \cite{helmberg1996interior}. 

Despite their expressive power, we are in practice limited to solving  small and middle-size SDPs. As the size of problems and number of constraints grow, interior-point methods show poor scalability for large-scale problems. However, for special cases like the relaxation \eqref{basic_SDP}, the interior-point method tailored to this problem scales well, since the matrix determining the Newton direction can be efficiently constructed \cite{helmberg1996interior}. 

The bound from \eqref{basic_SDP} is not strong enough to be successfully used  within a branch and bound framework in order to solve larger Max-Cut problems to optimality. By adding additional equality or inequality constraints known as cutting planes, which are valid for all feasible solutions of \eqref{MC}, we strengthen the upper bound. One such class of cutting planes, called triangle inequalities, is obtained as is shown below. Observe that for an arbitrary triangle with vertices $i < j < k$ in the graph, any partition of vertices cuts either $0$ or $2$ of its edges. Moving to the $\{-1,1\}$ model, this leads to 
\begin{align*}
-X_{ij}-X_{ik}-X_{jk} \le 1, \\
-X_{ij}+X_{ik}+X_{jk} \le 1, \\
X_{ij}-X_{ik}+X_{jk} \le 1, \\
X_{ij}+X_{ik}-X_{jk} \le 1,
\end{align*}
(see \cite{helmberg1998solving} and \cite{rendl2010solving}). The inequalities are collected as $\B_{\TRI}(X) \le e$ and we end up with the following SDP relaxation:
\begin{align}\label{streng_SDP}
\tag{$\textrm{MC}_\TRI$}
\begin{split}
\textrm{maximize } \quad& \langle L, X \rangle \\
\textrm{subject to} \quad&\diag (X) = e \\
\quad& \B_{\TRI} (X) \le e\\
\quad& X \succeq 0.
\end{split}
\end{align}

To exploit a stronger relaxation of \eqref{streng_SDP}, the bound is further strengthened by using pentagonal and heptagonal inequalities, which belong to the family of hypermetric inequalities, valid for any matrix $X$ from the convex hull of rank-one matrices $xx^T$, for $x \in \{-1,1\}^n$. For every integer vector $b$ such that $e^Tb$ is odd, the inequality $\vert x^Tb \vert \ge 1$ holds for all $x \in \{-1,1\}^n$. Therefore, we have $\left\langle bb^T, xx^T\right\rangle \ge 1$. The hypermetric inequality specified by the vector $b$ is the inequality
$$
\left\langle bb^T, X\right\rangle \ge 1.  
$$
In \texttt{MADAM}, we consider the subset of hypermetric inequalities generated by choosing $b$ with $b_i \in \{-1,0,1\}$ and by fixing the number of non-zero entries to 3, 5, or 7. These correspond to triangle, pentagonal, and heptagonal inequalities, respectively, and we collect them as $\B(X) \le e$. We obtain the following strengthening of \eqref{streng_SDP}:
\begin{align}\label{streng_hyp_SDP}
\tag{$\textrm{MC}_\HYP$}
\begin{split}
\textrm{maximize } \quad& \langle L, X \rangle \\
\textrm{subject to} \quad&\diag (X) = e \\
\quad& \B(X) \le e\\
\quad& X \succeq 0.
\end{split}
\end{align}

There is a very large number of triangle, pentagonal, and heptagonal inequalities in \eqref{streng_hyp_SDP}. Solving this relaxation  directly is intractable even for moderate values of $n$. In \texttt{MADAM}, we iteratively identify a subset of the proposed cutting planes
by separating the most violated inequalities using the current approximate solution. 
There are $4\binom{n}{3}$ triangle inequalities, but for a given $X$, we can enumerate all of them and identify the most violated ones. 
Due to a large number of pentagonal and heptagonal inequalities, the computational overhead to evaluate all of them is prohibitive for larger instances and the separation has to be done heuristically.
We use the idea proposed in BiqBin,
where the separation problem is reformulated as a quadratic assignment problem of the form
\begin{align*}
\textrm{minimize } \quad& \langle H, PXP^T \rangle \\
 \quad& P \in \Pi,
\end{align*}
where $\Pi$ is the set of all $n \times n$ permutation matrices and $H$ determines the type of pentagonal or heptagonal inequality specified by the vector $b$. The problem is then approximately solved by using simulated annealing to obtain inequality with a potentially large violation. The reader is referred to \cite{BiqBin:2020} for more details.

\end{section}

\begin{section}{Other solution approaches based on semidefinite programming}
In this section, we summarize other approaches for solving the Max-Cut problem to optimality, especially the BiqMac \cite{rendl2010solving}, BiqBin \cite{BiqBin:2020}, and BiqCrunch \cite{krislock2017biqcrunch} solvers, which are widely considered as the most powerful solvers for unconstrained binary quadratic  problems. 

\begin{subsection}{BiqMac and BiqBin} 
The starting point of BiqMac is the strengthened SDP relaxation \eqref{streng_SDP}.
By dualizing only the triangle inequality constraints, the authors obtain the convex nonsmooth partial dual function
$$
f(\gamma) = e^T\gamma + \max_{\substack{\textrm{diag}(X) \ = \ e,\\ X \succeq 0}}\langle L - \mathcal{B}_{\TRI}^T(\gamma), X \rangle,
$$
where $\gamma$ is the nonnegative dual variable. 
Evaluating the dual function and computing a subgradient amounts to solving an SDP of the form \eqref{basic_SDP}, which can be efficiently computed using an interior-point method tailored to this problem.  The approximate minimizer of the dual problem
\begin{align*}
\textrm{minimize } \quad& f(\gamma) \\
\textrm{subject to} \quad& \gamma \ge 0
\end{align*}
is then computed using the bundle method \cite{kiwiel1989survey}. In BiqBin, the bound is further strengthened by also dualizing pentagonal and heptagonal inequalities, i.e.\ the bundle method is used to minimize the partial dual function of \eqref{streng_hyp_SDP}. 

\end{subsection}

\begin{subsection}{BiqCrunch}
Due to the diagonal constraint in \eqref{streng_SDP}, it can easily be proven that all feasible matrices of this problem satisfy 
the condition $n^2 - \Vert X \Vert_F^2 \ge 0$, in which the equality holds if and only if the matrix has rank 1. The idea of BiqCrunch is to add a multiple of this quadratic regularization term to the objective function. They obtain the following regularized problem:
\begin{align}\label{regularized_SDP}
\begin{split}
\textrm{maximize } \quad& \langle L, X \rangle +\frac{\alpha}{2}\left(n^2 - \Vert X \Vert_F^2\right) \\
\textrm{subject to} \quad&\diag (X) = e \\
\quad& \B_{\TRI} (X) \le e\\
\quad& X \succeq 0.
\end{split}
\end{align}
By using quadratic regularization, the simplified dual problem
\begin{align*}
\textrm{minimize } \quad& \frac{1}{2\alpha}\left\Vert \left[\Diag(y) + \B_{\TRI}^T(z) - L \right]_+\right\Vert_F^2 + e^Ty + e^Tz + \frac{\alpha}{2}n^2\\[0.5em]
\textrm{subject to} \quad& y \textrm{ free}, \ z \ge 0
\end{align*}
produces a family of upper bounds dependent on the penalty parameter $\alpha$. Note that the objective function is convex and differentiable, see \cite{krislock2014improved}. The function value and gradient are evaluated by computing partial spectral decomposition. This enables that for a fixed $\alpha$, the dual function is minimized using the L-BFGS-B algorithm \cite{byrd1995limited} from the family of quasi-Newton methods. Reducing the penalty parameter $\alpha$ increases the tightness of the upper bound. This is due to the decreased impact of the regularization term in the primal objective function of \eqref{regularized_SDP}. 
\end{subsection}

\end{section}	

\begin{section}{Alternating direction method of multipliers}
Several algorithmic alternatives to interior-point methods have been proposed in the literature to solve semidefinite programs \cite{burer2003nonlinear, fischer2006computational, malick2009regularization}. The alternating direction method of multipliers (ADMM) has been studied over the last decades due to its wide range of applications in a number of areas and its capability of solving large-scale problems \cite{boyd2011distributed, povh2006boundary, wen2010alternating}. ADMM is an algorithm that solves convex optimization problems by breaking them into smaller pieces, which are easier to handle. This is achieved by an alternating optimization of the augmented Lagrangian function. 
The  method deals with the following problems with only two blocks of functions and variables:
\begin{align}\label{admm2block}
\begin{split}
\textrm{minimize }  \quad& f_1(x_1) + f_2(x_2) \\ 
\textrm{ subject to } \quad& A_1x_1 + A_2x_2 = b \\
\quad& x_1 \in \mathcal{X}_1, \  x_2 \in \mathcal{X}_2. 
\end{split}
\end{align}
Let 
$$
\mathcal{L}_{\rho}(x_1,x_2,y) = f_1(x_1) + f_2(x_2) + y^T\left(  A_1x_1 + A_2x_2 - b \right) + \frac{\rho}{2}\left\Vert  A_1x_1 + A_2x_2 - b \right\Vert_2^2
$$
be the augmented Lagrangian function of \eqref{admm2block} with the Lagrange multiplier $y$ and a penalty parameter $\rho > 0$. Then ADMM consists of iterations
\begin{align}\label{admm2block_iter}
\begin{split}
x_1^{k+1} &= \argmin_{x_1 \in \mathcal{X}_1} \mathcal{L}_{\rho}(x_1,x_2^k, y^k)\\[0.2em]
x_2^{k+1} &= \argmin_{x_2 \in \mathcal{X}_2} \mathcal{L}_{\rho}(x_1^{k+1},x_2, y^k)\\[0.2em]
y^{k+1} &= y^k + \rho\left( A_1x_1^{k+1} + A_2x_2^{k+1} - b \right).
\end{split}
\end{align}
For details the reader is referred to \cite{boyd2011distributed} and references therein. 

ADMM has already been successfully applied for solving semidefinite relaxations of combinatorial optimization problems. 
The Boundary Point Method \cite{povh2006boundary} is an augmented Lagrangian method applied to the dual SDP in standard form. In the implementation, alternating optimization is used since only one iteration is performed in the inner loop. The method is currently one of the best algorithms for computing the theta number  of a graph \cite{lovasz1979shannon}. However, the drawback of the method is that it can only solve equality constrained semidefinite programs. Wen et~al.\ \cite{wen2010alternating} present ADMM for general semidefinite programs with equality and inequality constraints.  Compared to our approach, the drawback of their method is in solving the quadratic program over the nonnegative orthant when computing the dual variable corresponding to inequality constraints. Furthermore, they use multi-block ADMM for which the proof of convergence is not presented. 

\begin{subsection}{ADMM method for \eqref{streng_hyp_SDP}}
By introducing the slack variable $s$,  the problem \eqref{streng_hyp_SDP} is equivalent to
\begin{align}\label{streng_SDP_slack}
\begin{split}
\textrm{maximize } \quad& \langle L, X \rangle \\
\textrm{subject to} \quad&\diag (X) = e \\
\quad& \B (X) +s =  e\\
\quad& X \succeq 0, \ s \ge 0.
\end{split}
\end{align}
One can easily verify that its dual problem can be written as
\begin{align} \label{streng_SDP_slack_dual}
\begin{split}
\textrm{minimize }  \quad& e^Ty + e^Tt \\ 
\textrm{ subject to } \quad& L-\Diag(y) -\B^T(t) + Z = 0 \\ 
\quad& u-t = 0  \\
\quad& y, t \textrm{ free}, \ Z \succeq 0, \ u \ge 0, 
\end{split}
\end{align}
where $y$ and $t$ are dual variables associated respectively with the equality constraints, whereas $u$ and $Z$ are dual multipliers to the conic constraints.  

Let $m$ denote the number of hypermetric inequalities considered in \eqref{streng_hyp_SDP}. For fixed $\rho > 0$, consider the augmented Lagrangian $L_{\rho}$ for \eqref{streng_SDP_slack_dual}:
\begin{align*}
L_{\rho}(y,t,Z,u;X,s) =&  \ e^Ty + e^Tt  + \langle L-\Diag(y)-\B^T(t)+Z, X \rangle \ +\\[0.5em]
&\ \frac{\rho}{2}\left\Vert L-\Diag(y)-\B^T(t)+Z\right\Vert^2 _F+  \langle u-t, s \rangle  + \frac{\rho}{2}\left\Vert u-t\right\Vert^2 _2\\[0.5em]
=& \ e^Ty + e^Tt  + \frac{\rho}{2}\left\Vert L-\Diag(y)-\B^T(t)+Z + X/ \rho\right\Vert^2 _F \ +\\[0.5em]
& \ \frac{\rho}{2}\left\Vert u-t + s/ \rho \right\Vert^2 _2 - \frac{1}{2\rho}\left\Vert X \right\Vert_F^2 - \frac{1}{2\rho}\left\Vert s  \right\Vert^2_2.
\end{align*}
The alternating direction method of multipliers for problem \eqref{streng_SDP_slack_dual} consists of the alternating minimization of $L_{\rho}$ with respect to one variable while keeping others fixed to get $y$, $Z$, $t$, and $u$. Then the primal variables $X$ and $s$ are updated using the following rules:
\begin{subequations}\label{admm_maxcut}
\begin{align}
y^{k+1} &= \argmin_{y \in \mathbb{R}^n} \ L_{\rho}(y,t^k,Z^k,u^k;X^k,s^k) \label{y}\\[0.5em]
t^{k+1} &= \argmin_{t  \in \mathbb{R}^m} \ L_{\rho}(y^{k+1},t,Z^k,u^k;X^k,s^k)\label{t}\\[0.5em]
Z^{k+1} &= \argmin_{Z \succeq 0} \ L_{\rho}(y^{k+1},t^{k+1},Z,u^k;X^k,s^k)\label{Z} \\[0.5em]
u^{k+1} &= \argmin_{s  \in \mathbb{R}^m_+} \ L_{\rho}(y^{k+1},t^{k+1},Z^{k+1},u;X^k,s^k)\label{u}\\[0.5em]
X^{k+1} &=  X^k + \rho\left(L-\Diag(y^{k+1})-\B^T(t^{k+1}) + Z^{k+1} \right)\label{X}\\[0.5em]
s^{k+1} &= s^k + \rho \left(u^{k+1}-t^{k+1}\right). \label{s}
\end{align}
\end{subequations}
Let us closely look at the subproblems \eqref{y}--\eqref{s}. 
The first order optimality condition for problem \eqref{y} is
$$
\nabla_y L_{\rho} = e +\rho y - \rho \diag(L-\B^T(t) + Z +X/ \rho) = 0.
$$
Hence the computation of $y$ is trivial
\begin{equation*}
y = -\frac{e}{\rho} + \diag(L -\B^T(t)+Z+ X/\rho). \label{update_y}
\end{equation*}
Similarly for \eqref{t}, we compute the gradient with respect to $t$
\begin{equation*}
\nabla_t L_{\rho} = e - s - \rho (u-t) - \rho \B(L-\Diag(y) - \B^T(t) + Z + X/\rho) = 0.
\end{equation*}
Therefore, $t$ is the solution of the following linear system 
\begin{equation*}
(\B\B^T + I)t = \B(L-\Diag(y)+Z+X/\rho) +u + \frac{s-e}{\rho}. \label{BBT}
\end{equation*}
Note that the $\B\B^T +  I$ matrix is sparse and positive definite. Hence the above linear system can be efficiently solved using the sparse Cholesky factorization. 

By defining $M = L - \Diag(y) - \B^T(t) + X/\rho$, the subproblem \eqref{Z} can be formulated as
\begin{equation*}
\argmin_{Z \succeq 0} \left\Vert M+Z \right\Vert_F^2,
\end{equation*}
where the solution 
\begin{equation}
Z = (-M)_+ = -M_{-}
\end{equation}
is the projection of the $-M$ matrix onto the positive semidefinite cone. 

The subproblem \eqref{u} can be written as
\begin{equation*}
\argmin_{u \ge 0} \left\Vert u - t + s/\rho \right\Vert_2^2.
\end{equation*}
It asks for the nonnegative vector that is closest to $v = t -s/\rho$.  Hence the solution is 
\begin{equation*}
u = \max(0, v) = v_+,
\end{equation*}
the nonnegative part of vector $v$.

After all variables for the dual problem \eqref{streng_SDP_slack_dual} have been updated in alternating manner, we compute primal variables $X$ and $s$ using the expressions \eqref{X} and \eqref{s}. As already observed in the Boundary Point Method \cite{povh2006boundary}, the update rule for $X$ can be simplified and also computed from the spectral decomposition of matrix $M$. Using \eqref{X}, we get
\begin{align*}
X^{k+1} &= X^k + \rho\left(L - \Diag(y^{k+1}) - \B^T(t^{k+1}) + Z^{k+1}\right) \\
&= \rho \left(L - \Diag(y^{k+1}) - \B^T(t^{k+1})  + X^k/\rho + Z^{k+1}\right) \\
&= \rho(M - M_-) = \rho M_+.
\end{align*}
Similarly, we can simplify the formula for \eqref{s}
\begin{align*}
s^{k+1} &= s^k + \rho (u^{k+1}-t^{k+1}) \\
&= \rho \left( u^{k+1} - (t^{k+1} - s^k/\rho)\right)\\
 &= \rho (v_+ - v) = -\rho v_-.
\end{align*}
Hence by construction, the matrix $X$ is positive semidefinite and the vector $s$ is nonnegative. 

The overall complexity of one iteration of the method is solving a linear system with matrix $\B\B^T + I$ and computing the partial eigenvalue decompostion of $M$. Compared to interior-point methods where the coefficient matrix changes in each iteration, matrix $\B\B^T + I$ remains constant throughout the algorithm and its factorization can be cached at the beginning to efficiently solve the linear system in each iteration. 

The difference between Algorithm \ref{admm_alg}  and the method proposed by Wen et~al.\ \cite{wen2010alternating} lies in the update rule of the dual variable $t$ corresponding to inequality constraints. The authors directly apply ADMM on the dual of \eqref{streng_hyp_SDP}. In our approach, we introduce the slack variable $s$ resulting in an unconstrained optimization problem for variable $t$ in \eqref{t}. This reduces the overall complexity of the algorithm by eliminating the need to solve a convex quadratic program of order $m$ over the nonnegative orthant, especially since multiple hypermetric inequalities are added to strengthen the bound. Instead, a sparse system of linear equations is solved and a projection onto the nonnegative orthant is used. 

\end{subsection}

\begin{subsection}{Implementation}
The above update rules ensure that during the algorithm, nonnegativity of vectors $u$ and $s$, and conic constraints for matrices $X$ and $Z$ are maintained, as well as complementarity conditions $u^Ts = 0$ and $ZX = 0$. Hence, once primal and dual feasibility are reached, the method converges to the optimal solution.
To measure the accuracy of primal and dual feasibility, we use
\begin{align*}
r_P &= \frac{\Vert \diag(X) - e\Vert_2 + \Vert \max\left(\B(x) - e, 0\right)\Vert_2}{1 + \sqrt{n}},\\[0.5em]
r_D &= \frac{\Vert L - \Diag(y) - \B^T(t) + Z \Vert_F + \Vert u - t \Vert_2 }{1 + \Vert L \Vert_F}.
\end{align*}
We terminate our algorithm when
$
\max\{r_P,  r_D\} < \varepsilon,
$
for prescribed tolerance $\varepsilon > 0$. 

The performance of the method is dependent on the choice of the penalty parameter $\rho$.  Numerical experiments show that for the problems we consider, the starting value of $\rho = 1$ or $\rho = 1.6$ is a good choice and the value is dynamically tuned during the algorithm in order to improve the practical convergence. A simple strategy to adjust the value of $\rho$ is observing the residuals: 
\begin{align}
\label{update_ro}
\begin{split}
\rho^{k+1} = 
\begin{cases} 
      \tau\rho^k & \textrm{ if } \log\left(\frac{r_D}{r_P}\right)> \mu\\[0.5em]
      \rho^k/\tau & \textrm{ if } \log\left(\frac{r_P}{r_D}\right)> \mu \\[0.5em]
      \rho^k & \textrm{ otherwise,}
   \end{cases}
   \end{split}
\end{align}
for some parameters $\mu$ and $\tau$. In our numerical tests, we use $\mu = 0.5$ and $\tau = 1.001$. The idea behind this penalty parameter update scheme is trying to keep the primal and dual residual norms in the same order of magnitude as they both converge to zero. 

The computational time of our ADMM method is essentially determined by the number of partial eigenvalue decompositions and the efficiency of the sparse Cholesky solver, since these are the most computationally expensive steps. For obtaining positive eigenvalues and corresponding eigenvectors, we use the LAPACK \cite{anderson1999lapack} routine \texttt{DSYEVR}.  For factoring the $\B\B^T+I$ matrix and then performing backsolves to get $t$, we use the sparse direct solver from CHOLMOD \cite{chen2008algorithm}, a high performance library for the sparse Cholesky factorization. CHOLMOD is part of the SuiteSparse linear algebra package \cite{davis2011university}. 

\end{subsection}

\vspace{1em}
In the following two subsections, we elaborate on two potential issues when using ADMM, and how to resolve them. These are obtaining a safe upper bound, which can be used within the B\&B algorithm, and the convergence of multi-block ADMM. 

\begin{subsection}{Safe upper bound}
To safely use the proposed upper bound within the B\&B algorithm, we need a certificate that the value of the dual function is indeed a valid upper bound for the original problem \eqref{maxcut}. 
To achieve this, the quadruplet $(y,t,Z,u)$ has to be dual feasible, i.e.\ by assigning $t \leftarrow u$, the equation $$L - \Diag(y) - \B^T(u) + Z = 0$$ has to be satisfied. 
However, since we only approximately solve the primal-dual pair of semidefinite programs to some precision, the dual feasibility is not necessarily reached when the algorithm terminates. Note that variable $y$ is unconstrained, whereas the conic conditions on $u$ and $Z$ are satisfied by construction. In the following, we describe the post-processing step we do after each computation of the bound.

\begin{proposition}\label{proposition_safe_bound}
Let $y$, $t$, $u$, and $Z$ be the output variables computed with iteration scheme \eqref{y}~--~\eqref{s}.  Let $\lambda_{\min}$ denote the smallest eigenvalue of matrix $\hat Z:= \Diag(y) + \B^T(u) - L$. If $\lambda_{\min} \ge 0$, then the value $e^Ty + e^Tu$ is a valid upper bound for problem \eqref{maxcut}. Otherwise the value $e^T\hat{y} + e^Tu$, where $\hat{y} = y -\lambda_{\min} e$, provides an upper bound on the largest cut. Furthermore, it always holds that the value $e^T\hat{y} + e^Tu$ is larger than $e^Ty + e^Tu$. 
\end{proposition}
\begin{proof}
After the ADMM method terminates, and by assigning $t \leftarrow u$, the quadruplet $(y,t,Z,u)$ satisfies all the constraints of \eqref{streng_SDP_slack_dual} within machine accuracy, only the linear constraint
$$
L - \Diag(y) - \B^T(u) +Z=0
$$
may not be satisfied.
By replacing $Z$ with $\hat Z$, we satisfy this linear constraint, but the positive semidefiniteness of $\hat Z$ might be violated.
 If the smallest eigenvalue $\lambda_{\min}$  of $\hat Z$ is nonnegative, then the new quadruplet $(y,t,\hat{Z},u)$ is feasible for \eqref{streng_SDP_slack_dual} and the $e^Ty + e^Tu$ value provides an upper bound on the size of the largest cut. 
 
If $\lambda_{\min}$  is negative,  we adjust the matrix $\hat Z$ to be positive semidefinite by using
$$
\hat Z \leftarrow \hat Z -\lambda_{\min}  I.
$$
To maintain dual equality constraints, we correct the unconstrained variable as $\hat{y} = y - \lambda_{\min}  e$. It is obvious that the value $e^T\hat{y} + e^Tu = e^Ty + e^Tu - n\lambda_{\min}$ is larger than $e^Ty + e^Tu$. 
\end{proof}

We summarize the ADMM-based algorithm for solving the SDP relaxation \eqref{streng_hyp_SDP} in Algorithm \ref{admm_alg}.

\begin{algorithm}[h]
	\footnotesize
	\caption{\small ADMM method for solving SDP relaxation \eqref{streng_hyp_SDP} of Max-Cut }
	\label{admm_alg}
	\KwOut{Upper bound on the largest cut of a graph.} 
	Select $\rho > 0$, $\varepsilon > 0$\;
    $k= 0$, $u^k = 0$, $s^k = 0$, $X^k = 0$, $Z^k = 0$\;
    Compute sparse Cholesky factorization of matrix $\B\B^T + I$\;
    \For{$k = 0, 1,\ldots$}{
    		Compute $y^{k+1} = -\frac{e}{\rho} + \diag(L +Z^k+ X^k/\rho)$ \;
    		Solve for $t^{k+1}$ by using cached factorization: $(\B\B^T + I)t^{k+1} = \B(L+Z^k+X^k/\rho) +u^k + \frac{s^k-e}{\rho}$\;
    		$v = t^{k+1} - s^k/\rho$\;
    		$M = L - \Diag(y^{k+1}) - \B^T(t^{k+1}) + X^k/\rho$\;
    		$Z^{k+1} = -M_-$\;
    		$u^{k+1} = v_+$\;
    		$X^{k+1} = \rho M_+$\;
    		$s^{k+1} = -\rho v_-$\;
    		Compute $\delta = \max\{r_P, r_D\}$\;
         \If{$\delta < \varepsilon$}{break\;}
		Update penalty parameter $\rho$ according to \eqref{update_ro}.
	}
	\emph{Post-processing:}\\
	$\hat{Z} = \Diag(y) + \B^T(t) - L$\;
	Compute the smallest eigenvalue $\lambda_{\min}(\hat{Z})$\;
	\If{ $\lambda_{\min}(\hat{Z}) \ge 0$}{return $e^Ty + e^Tu$\;}
	\Else{return $e^Ty + e^Tu - n\lambda_{\min}(\hat{Z}) $\;}	        
\end{algorithm}

\end{subsection}

\begin{subsection}{Convergence of the method }
It has been recently shown \cite{chen2016direct} that multi-block ADMM is not necessarily convergent. In the theorem presented below, we show that due to the special structure of operators in the case of semidefinite relaxation of Max-Cut,  we can achieve a convergent scheme by reducing it to a 2-block method. For the sake of completeness, we include the proof of convergence. 
We also note that Chen et~al.\ in \cite{chen2016direct} prove the same result in a more general setting.

\begin{theorem}
The sequence $\left\{ \left( X^k, s^k, y^k, t^k, Z^k, u^k \right)\right\}$ generated by Algorithm \ref{admm_alg} from any starting point $\left( X^0, s^0, y^0, t^0, Z^0, u^0 \right)$ converges to solutions $\left( X^*, s^*\right)$, $\left(y^*, t^*, Z^*, u^* \right)$ of the primal-dual pair of semidefinite programs \eqref{streng_SDP_slack} and \eqref{streng_SDP_slack_dual}. 
\end{theorem}

\begin{proof}
The convergence of our multi-block ADMM is guaranteed due to the orthogonality relations of the operators $\diag$ and $\B$ and their adjoints:
\begin{align}\label{orthogonality}
 \B\left(\Diag(y)\right) = 0, \ \ \diag\left(\B^T(t)\right) = 0, \ \textrm{ for any vectors } y \in  \mathbb{R}^n \textrm{ and }  t \in \mathbb{R}^m.
\end{align}
In this case the multi-block ADMM reduces to a special case of the original method \eqref{admm2block_iter}. To see this, note that orthogonality in \eqref{orthogonality} implies that the first order optimality conditions for variables $y$ and $t$ in \eqref{y} and \eqref{t} reduce to 
\begin{align*}
\nabla_y L_{\rho} &= e +\rho y - \rho \diag(L+ Z +X/ \rho) = 0\\
\nabla_t L_{\rho} &= e - s - \rho (u-t) - \rho \B(L - \B^T(t) + Z + X/\rho) = 0,
\end{align*}
meaning that $y$ and $t$ are independent and are jointly minimized by regarding $(y,t)$ as one variable. Similarly, update rules for variables $Z$ and $u$, as well as for primal pair $X$ and $s$, ensure that they can also be jointly minimized. By separately regarding $(Z,u)$ and $(X,s)$ as one variable, the iterate scheme \eqref{y}~--~\eqref{s} can be written as:
\begin{align*}
\left(y^{k+1},t^{k+1}\right) &= \argmin_{(y,t) \in \mathbb{R}^n \times \mathbb{R}^m} \ L_{\rho} \left( (y,t), (Z^k,u^k);(X^k,s^k) \right)\\[0.5em]
\left(Z^{k+1}, u^{k+1} \right) &= \argmin_{(Z,u) \in \mathcal{S}_n^+ \times \mathbb{R}^m_+} \ L_{\rho}\left( (y^{k+1},t^{k+1}),(Z,u);(X^k,s^k)\right) \\[0.5em]
\left(X^{k+1}, s^{k+1} \right) &=  \left(X^{k}, s^{k} \right) + \rho\left(L-\Diag(y^{k+1})-\B^T(t^{k+1}) + Z^{k+1}, u^{k+1}-t^{k+1} \right).
\end{align*}
Thus the convergence of Algorithm \ref{admm_alg} is implied by the analysis in \cite{wen2010alternating}, which looks at 2-block ADMM as a fixed point method. 
\end{proof}

\end{subsection}

\end{section}

\begin{section}{Branch and bound}\label{sec:BB}
The prominent bounds obtained by using the ADMM method for solving \eqref{streng_hyp_SDP} motivate us to use this method within a branch and bound framework. We call the outcome of this work the \texttt{MADAM} solver -- \texttt{M}ax-Cut \texttt{A}lternating \texttt{D}irection \texttt{A}ugmented Lagrangian \texttt{M}ethod.  

The Branch and Bound algorithm (B\&B) is one of the standard enumerative methods for computing  the global optimum of an NP-hard problem. It consists of the following ingredients:
\begin{itemize}
\item the bounding procedure, which provides an upper bound for the optimum value for each instance of a problem;
\item the branching procedure, which splits the current problem into more problems of smaller dimensions by fixing some variables;
\item a heuristic for generating a feasible solution providing a lower bound.
\end{itemize}

\paragraph{Bounding routine.}
The bounding routine of \texttt{MADAM} applies the ADMM algorithm, described in Algorithm \ref{admm_alg}, which for a given set of triangle, pentagonal, and heptagonal inequalities minimizes the dual function  \eqref{streng_SDP_slack_dual}. In order to obtain a tight upper bound,  we use the cutting-plane approach where multiple hypermetric inequalities are iteratively added and purged after each computation of the upper bound. First, the optimal solution of the basic semidefinite relaxation \eqref{basic_SDP} is computed using the interior-point method.  Then we separate the initial set of triangle inequalities. 
After the minimizer of problem \eqref{streng_SDP_slack_dual} is obtained, we purge all inactive constraints and new violated triangles are added. 
The problem with an updated set of inequalities is solved and the process is iterated as long as the decrease of the upper bound is sufficiently large. 
After separating the most violated triangle inequalities, we add pentagonal and heptagonal inequalities in order to further decrease the upper bound. We monitor the maximum violation of triangle inequalities $r_{\TRI}$ and as soon as the number is sufficiently small, we use the heuristic from Section 2 to add some strongly violated pentagonal inequalities to the relaxation. Similarly, as the maximum violation of pentagonal inequalities $r_{\PENT}$ drops below some threshold, new heptagonal inequalities are separated and added to the relaxation. In our numerical tests, we use the thresholds $r_{\TRI} < 0.2$ and $r_{\PENT} < 0.4$, and each time separate at most $10\cdot n$ violated triangle inequalities. 
For pentagonal and heptagonal inequalities, a different strategy is used. The initial number of these cutting planes is small, but each time the node is not pruned, this number is increased by 200.    
A description of our bounding routine is in Algorithm \ref{bounding_procedure}.

\begin{algorithm}[h]
	\footnotesize
	\caption{\small Semidefinite bounding procedure of \texttt{MADAM}}
	\label{bounding_procedure}
	solve the basic SDP relaxation \eqref{basic_SDP}\;
	separate initial set of triangle inequalities\;
	\While{upper bound decreases significantly}{
		use ADMM method to obtain approximate maximizer $X$ of \eqref{streng_hyp_SDP}\;
		remove hypermetric inequalities that are not active\;
		add new hypermetric inequalities that are violated\;
	}
\end{algorithm}

The warm-start strategy is the standard approach taken with ADMM methods to ensure faster convergence. After solving \eqref{streng_hyp_SDP} approximately at each iteration, the inactive hypermetric inequalities are removed and new hypermetric inequalities are added. This yields a new relaxed instance with a different set of inequality constraints. Thus, we can save the variables from the output of the ADMM at the $k$\nobreakdash-th iteration, and then use them to initialize the variables for the ADMM at the $(k+1)$\nobreakdash-th iteration. The previous ADMM iterates give a good enough approximation to result in far fewer iterations to compute the updates than if the method were started at zero or some other initial values. This is especially the case when the variables are near optimal and the method has almost converged. Furthermore, this strategy works particularly well in our case since only the variables $s$ and $u$ corresponding to cutting planes need to be adjusted to fit with the new operator $\B$ associated with a different set of inequality constraints. We keep the values $s_i$ and $u_i$ the same if the $i$-th inequality stays, otherwise we initialize the values with zero.

In order to improve the performance of \texttt{MADAM}, we can stop the bounding routine when we detect that we will not be able to prune the current B\&B node. BiqCrunch terminates the bounding routine if the difference between consecutive bounds is less than some prescribed parameter and if the gap between the upper and lower bounds is still large. In \texttt{MADAM}, we borrow an idea from BiqMac. After some cutting plane iterations, we make a linear forecast to decide whether it is worth doing more iterations. If the gap cannot be closed, we terminate the bounding routine, branch the current node, and start evaluating new subproblems.

\paragraph{Branching rules.} In \texttt{MADAM}, we use two branching strategies which are based on the approximate solution matrix $X$ obtained from Algorithm \ref{admm_alg}. Once the ADMM method terminates, the last column of $X$, with the exception of the last element, is extracted. Due to the diagonal constraint and positive semidefiniteness of feasible matrices of \eqref{streng_hyp_SDP}, all the entries lie in the interval $[-1,1]$. 
By using the transformation $z = \frac{1}{2}(x + e)$, a vector with entries in $[0,1]$ is obtained. We can follow different strategies based on the information we get from $z$.
Similarly to BiqCrunch and BiqBin, the decision on which variable $z_i$ to branch on is based on the following two strategies:
\begin{enumerate}[i]
\item most-fractional: we branch on the vertex $i$ for which the variable $z_i$ is closest to 0.5;
\item least-fractional: we branch on the vertex $i$ for which the variable $z_i$ is furthest from 0.5;
\end{enumerate}  

With this, we create two new subproblems, one where $z_i$ is fixed to 0, and the other where it is fixed to 1. The subproblems correspond to nodes in the B\&B tree and are added to the priority queue of unexplored problems. 
Priority is based on the upper bound  computed  by the ADMM method. When selecting the next subproblem, a node with the largest upper bound is evaluated first.

\paragraph{Generating feasible solutions.} For generating high quality feasible solutions of the Max-Cut problem, we use the optimum solution $X$ of \eqref{streng_hyp_SDP} in the Goemans-Williamson rounding hyperplane technique \cite{goemans1995improved}. By using factorization $X = V^\top V$ with column vectors $v_i$ of $V$ and selecting some random vector $r$, one side of the cut is obtained as 
$
\{i \mid v_i^\top r \ge 0 \}. 
$
The cut vector $x$ obtained from this heuristic is then further improved by flipping the vertices. Note that in our case the factorization of $X$ is already available from the bounding routine, which computes the eigenvalue decomposition of matrix $M$.  To summarize, for generating good cuts, we use the following scheme:
\begin{enumerate}[i]
\item Use the Goemans-Williamson rounding hyperplane technique to generate cut vector $x$ from minimizer of \eqref{streng_hyp_SDP}. 
\item The cut $x$ is locally improved by checking all possible moves of a single vertex to the opposite partition block. 
\end{enumerate}
The process is repeated several times with different random vectors $v$, due to its low computational cost (in practice we repeat it $n$ times, where $n$ is the size of the input graph).
Numerical experiments show that by using the proposed heuristic, the optimum solution is usually found already in the root node. 

\paragraph{Strategy for faster enumeration of the B\&B tree.} To obtain tight upper bounds on maximum cuts, we strengthen the basic SDP relaxation with a subset of hypermetric inequalities. Adding these cutting planes to the model and solving the relaxation in each B\&B node is computationally expensive and not necessary, especially if one can not prune the node. We use the strategy proposed in BiqBin, where we check (before including the hypermetric inequalities)  whether we will be able to prune the current node, or whether it is better to branch. 

Firstly, in the root node, we cache the bound $OPT_{\SDP}$ of the basic SDP relaxation \eqref{basic_SDP} and compute the bound $OPT_{\HYP}$ by iteratively adding violated hypermetric inequalities. Let
$$
\diff = OPT_{\SDP} - OPT_{\HYP}
$$
denote the difference between optimal values of both relaxations
and let $lb$ denote the current lower bound. 
Secondly, at all other nodes, we only compute the basic SDP relaxation. If the condition
\begin{align}
\label{diff}
OPT_{\SDP} \le lb + \diff + 1
\end{align}
is satisfied, meaning we are close enough to the lower bound, we estimate that by
adding the cutting planes to compute the tighter bound $OPT_{\HYP}$, we will prune the node. This idea helps efficiently traverse the B\&B tree, and we only invest time into the bounding routine when it is really needed. Numerical experiments show that overall, this strategy produces more B\&B nodes than necessary, but the performance of the algorithm is immensely improved, since in the first few levels of the B\&B tree, only the basic SDP relaxation needs to be computed.  

\end{section}

\begin{section}{Numerical results - serial algorithm}
We compare our serial version of \texttt{MADAM} with three state-of-the-art semidefinite based exact solvers for Max-Cut, BiqMac \cite{rendl2010solving}, BiqCrunch \cite{krislock2017biqcrunch}, and BiqBin \cite{BiqBin:2020}. All solvers are written in the C programming language. Several dense graphs available in the BiqMac library \cite{wiegele2007biq} with 100 vertices and integer edge weights are used to test the algorithms.

First, we take the pw05\_100.0 problem and plot the convergence curves for the bounding routines in the BiqCrunch, BiqMac, BiqBin, and \texttt{MADAM} solvers in the root nodes of the corresponding B\&B trees. Figure \ref{compare_bound} depicts the decrease of the dual function values in the course of the bound computation. Note that we have moved the BiqMac curve slightly to the right to make a better distinction between the BiqBin and the BiqMac curves, since they were strongly overlapping at the beginning. We also omit the first part of the BiqCrunch curve, since at the beginning, the penalty parameter $\alpha$ is large, and the bounding routine produces weak bounds. 

\begin{figure}[ht!]
	\centering
\includegraphics[scale=0.6]{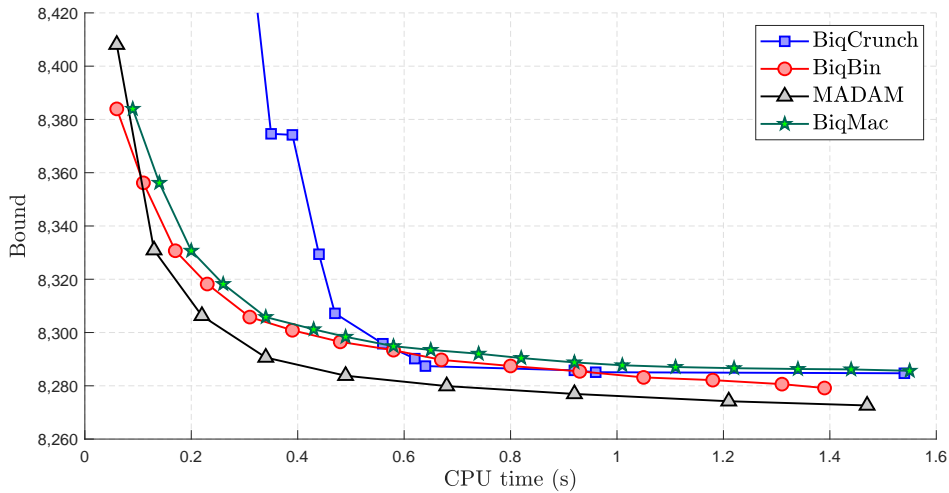}
\caption{Comparison of the bounding routine of BiqCrunch, BiqMac, BiqBin, and \texttt{MADAM} in the root node on the pw05\_100.0 problem.}
\label{compare_bound}
\end{figure}

The improvements of the serial version of \texttt{MADAM} over the other three solvers are due to using the SDP relaxation that produces a tighter upper bound, and in using the ADMM method instead of the bundle method or the quasi-Newton method. BiqMac and BiqCrunch use only triangle inequalities to strengthen the basic SDP relaxation, whereas BiqBin and \texttt{MADAM} also add  pentagonal and heptagonal inequalities. Consequently, the bound on the maximum cut is tighter. 
The exact minimizer of the nonsmooth partial dual function of BiqBin is difficult to reach using the bundle method. On the other hand, it is often the case that ADMM converges to modest accuracy (which is sufficient for our application) in a small number of iterations.
Although BiqBin and \texttt{MADAM} use the same SDP relaxation to compute the upper bound, there is a fundamental difference in the routine that separates new violated inequalities.
After using the ADMM method, we obtain the approximate maximizer $X$ of \eqref{streng_hyp_SDP}. This matrix is used to add new cutting planes to the relaxation. However, this information is not available from the bundle method. Instead, a convex combination of the bundle matrices is used as the input of the separation routine. For this reason, the \texttt{MADAM} solver is more successful in identifying promising cutting planes and consequently the upper bound is tighter.  
Furthermore, in our approach we reduce the complexity of ADMM by introducing a slack variable. Without it, the main cost in each iteration would consist of the projection onto the positive semidefinite cone, and solving a quadratic program over the nonnegative orthant with $m$ variables, where $m$ denotes the number of cutting planes. Instead of solving a quadratic program, we need to compute the solution of the sparse linear system and take the nonnegative part of the linear combination of the obtained vector and slack variable. This refinement further increases the performance of ADMM.

Next, we compare the running times of the solvers. Let $lb$ denote the best lower bound given by the heuristic, and $ub$ is the upper bound from the bounding procedure. Using the fact that we know that the optimal cut is an integer, we can prune the node of the branch and bound tree if the condition
\begin{equation*}
ub < lb + 1
\end{equation*}
is satisfied. For BiqMac, BiqCrunch, and BiqBin, we used their default settings. Each test instance is solved with both branching rules that the solvers provide. The one that produced the faster time of execution is used in the tables presenting the numerical results. For \texttt{MADAM}, the experiments show that the most-fractional rule dominates over the least-fractional rule, i.e.\ produces faster execution times for most of the graphs, and is therefore set as the default strategy.
We also run the algorithm with different penalty parameter values $\rho$ found in the literature \cite{wen2010alternating}. The value of the penalty parameter is set to $1.6$ for graphs from families g05 and pm1d, while for graphs from families w05, w09, pw05, and pw09 this value is decreased to $1$, since we have found this works well in our code. 
All the computations are done on a cluster consisting of 24-core nodes featuring two Intel Xeon E5-2680V3 running at 2.5 GHz with 64 GB of RAM. The code is compiled against OpenBLAS and LAPACK. 
In Tables \ref{g05} to \ref{pw0}, we report  the needed number of B\&B nodes and running time to solve the problem to optimality. Observe that our solver completely outperforms other approaches. For many instances, we significantly reduce the running time of the current best solver. The reason for efficiency is the simplicity and speed of the proposed bounding routine.  

\begin{table}[!htbp]
\footnotesize
\centering
\begin{tabular}{|c|r r|r r|r r|r r|r r|}
\bottomrule
&
  \multicolumn{2}{c|}{BiqMac} &
  \multicolumn{2}{c|}{BiqCrunch} &
  \multicolumn{2}{c|}{BiqBin} &
  \multicolumn{2}{c|}{\texttt{MADAM} - serial} \\ \cline{2-9} 
  \multirow{-2}{*}{Problem} &
  Time &
  Nodes &
  Time &
  Nodes &
  Time &
  Nodes &
  Time &
  Nodes \\ \toprule \bottomrule
g05\_100.0                       & 555.16  & 531  & 355.58        & 329  &166.93 & 269  & \textbf{98.33} & 195  \\ \hline
g05\_100.1                       & 3547.17 & 3643 & 1933.51       & 1751  & 955.21 & 964 &  \textbf{494.10} & 705 \\ \hline
g05\_100.2                       & 115.87  & 127  & 125.57        & 95   & 52.76 & 83 &  \textbf{40.07}  & 43  \\ \hline
g05\_100.3                       & 1308.85 & 1215 & 602.59        & 643  & 324.56& 655 &  \textbf{129.60} & 497  \\ \hline
g05\_100.4                       & 71.03   & 69   & 44.35         & 33   & 18.30& 15 &  \textbf{9.71}  & 11   \\ \hline
g05\_100.5                       & 116.16  & 129  & 111.12        & 109  & 37.37 & 49 &  \textbf{28.63}  & 31  \\ \hline
g05\_100.6                       & 177.22  & 193  & 139.52        & 99   & 51.05 & 67 &  \textbf{29.52}  & 47  \\ \hline
g05\_100.7                       & 332.35  & 337  & 282.33        & 329  & 145.12 & 121 &  \textbf{75.31} & 73  \\ \hline
g05\_100.8                       & 291.28  & 275  & 198.62        & 163  & 72.40 & 73 &  \textbf{35.78}  & 67  \\ \hline
g05\_100.9                       & 321.10  & 277  & 232.88        & 229  &104.52 & 149 &  \textbf{47.34} & 101  \\ \toprule
\end{tabular}
\caption{CPU times (s) and B\&B nodes to solve g05 problems.}
\label{g05}
\end{table}

\begin{table}[!htbp]
\footnotesize
\centering
\begin{tabular}{|c|r r|r r|r r|r r|}
\bottomrule
&
  \multicolumn{2}{c|}{BiqMac} &
  \multicolumn{2}{c|}{BiqCrunch} &
   \multicolumn{2}{c|}{BiqBin} &
  \multicolumn{2}{c|}{\texttt{MADAM} - serial} \\ \cline{2-9} 
  \multirow{-2}{*}{Problem} &
  Time &
  Nodes &
  Time &
  Nodes &
  Time &
  Nodes &
  Time &
  Nodes \\ \toprule \bottomrule
pm1d\_100.0  & 917.88  & 867  & 821.70  & 1011 &  317.62& 507 & \textbf{138.36} & 287  \\ \hline
pm1d\_100.1 & 2191.20 & 2113 & 1302.31 & 1073 & 510.81 & 611 & \textbf{267.08} & 383 \\ \hline
pm1d\_100.2  & 1458.67 & 1357 & 1003.45 & 837  & 369.62 & 707 & \textbf{185.09} & 585  \\ \hline
pm1d\_100.3  & 298.95  & 285  & 273.41  & 281  &103.10 & 141 & \textbf{51.05}  & 81  \\ \hline
pm1d\_100.4  & 1188.23 & 1113 & 756.02  & 607  & 294.30 & 473 & \textbf{142.01} & 329  \\ \hline
pm1d\_100.5  & 260.25  & 255  & 294.19  & 211  & 90.58 & 73 & \textbf{57.65}  & 61  \\ \hline
pm1d\_100.6  & 248.40  & 217  & 182.62  & 127  & 72.72 & 117 & \textbf{36.80}  & 59  \\ \hline
pm1d\_100.7  & 80.00   & 109  & 104.40  & 103  & 36.80& 29 & \textbf{25.80}  & 17   \\ \hline
pm1d\_100.8  & 107.25  & 97   & 57.18   & 39   & 21.66 & 23 & \textbf{12.79}  & 21   \\ \hline
pm1d\_100.9  & 382.85  & 347  & 294.32  & 243  & 112.32 & 135 & \textbf{55.55} & 97  \\ \toprule
\end{tabular}
\caption{CPU times (s) and B\&B nodes to solve pm problems.}
\end{table}

\begin{table}[!htbp]
\footnotesize
\centering
\begin{tabular}{|c|r r|r r|r r|r r|}
\bottomrule
&
  \multicolumn{2}{c|}{BiqMac} &
  \multicolumn{2}{c|}{BiqCrunch} &
   \multicolumn{2}{c|}{BiqBin} &
  \multicolumn{2}{c|}{\texttt{MADAM} - serial} \\ \cline{2-9} 
  \multirow{-2}{*}{Problem} &
  Time &
  Nodes &
  Time &
  Nodes &
  Time &
  Nodes &
  Time &
  Nodes \\ \toprule \bottomrule
w05\_100.0      & 735.17         & 667  & 524.39           & 401  & 194.08 & 367 & \textbf{114.82} & 171  \\ \hline
w05\_100.1      & 182.31         & 187  & 157.68           & 95   & 69.11 & 61 &  \textbf{41.82} & 31   \\ \hline
w05\_100.2      & 137.86         & 139  & 104.17           & 63   & 43.36 & 31 & \textbf{27.14}  & 23   \\ \hline
w05\_100.3     & 532.43         & 473  & 383.18           & 439  & 155.61 & 213 & \textbf{120.45} & 125  \\ \hline
w05\_100.4     & 921.26         & 849  & 564.00           & 383  & 240.62 & 273 & \textbf{168.03} & 191  \\ \hline
w05\_100.5     & 969.22         & 861  & 532.27           & 337  & 199.36 & 147 & \textbf{150.84} & 155  \\ \hline
w05\_100.6    & 185.23         & 165  & 120.55           & 81   & 53.42 & 67 & \textbf{41.32}  & 45   \\ \hline
w05\_100.7    & 138.66         & 117  & 81.11            & 41   & 39.22 & 23 & \textbf{28.79}  & 13   \\ \hline
w05\_100.8     & 799.63         & 693  & 514.31           & 387  & 194.68 & 237 & \textbf{150.44} & 193  \\ \hline
w05\_100.9     & 98.13          & 97   & 58.88            & 31   & 23.88 & 19 & \textbf{17.70}  & 9  \\  \toprule \bottomrule
w09\_100.0    & 474.91         & 419  & 393.72           & 255  & 129.50 & 143 &\textbf{119.88} & 109  \\ \hline
w09\_100.1     & 2538.66        & 2553 & 2447.42 & 1667 & 829.93 & 811 & \textbf{635.93}        & 641 \\ \hline
w09\_100.2     & 1043.74        & 973  & 787.44           & 531  & 289.20 & 449 & \textbf{222.66} & 311  \\ \hline
w09\_100.3     & 1242.63        & 1151 & 1098.87          & 819  & 336.90 & 619 & \textbf{263.81} & 435  \\ \hline
w09\_100.4    & 575.05         & 541  & 575.57           & 361  & 194.62 & 183 & \textbf{138.72} & 155 \\ \hline
w09\_100.5    & 35.84          & 17   & 17.29            & 7    & \textbf{9.86} & 7 & 10.64  & 1    \\ \hline
w09\_100.6     &81.73 & 101  & 110.93           & 53   & \textbf{27.36} & 23 & 35.35           & 15   \\ \hline
w09\_100.7    & 318.65         & 269  & 287.99           & 187  & 87.26 & 83 & \textbf{76.59} & 75  \\ \hline
w09\_100.8    & 174.95         & 199  & 188.65           & 89   & 87.18 & 51 &\textbf{76.31} & 83   \\ \hline
w09\_100.9    & 504.25         & 449  & 358.77           & 263  & 119.36 & 149 & \textbf{113.38} & 207  \\ \toprule
\end{tabular}
\caption{CPU times (s) and B\&B nodes to solve w problems.}
\end{table}

\begin{table}[!htbp]
\footnotesize
\centering
\begin{tabular}{|c|r r|r r|r r|r r|}
\bottomrule
&
  \multicolumn{2}{c|}{BiqMac} &
  \multicolumn{2}{c|}{BiqCrunch} &
  \multicolumn{2}{c|}{BiqBin} &
  \multicolumn{2}{c|}{\texttt{MADAM} - serial} \\ \cline{2-9} 
  \multirow{-2}{*}{Problem} &
  Time &
  Nodes &
  Time &
  Nodes &
  Time &
  Nodes &
  Time &
  Nodes \\ \toprule \bottomrule
 pw05\_100.0   & 1657.05 & 1553 & 1155.51       & 999 & 487.22 & 805 & \textbf{274.04} & 629 \\ \hline
pw05\_100.1   & 481.93  & 427  & 384.57        & 327 & 162.50 & 243 & \textbf{101.95} & 145  \\ \hline
pw05\_100.2   & 709.90  & 687  & 454.90        & 359 & 194.99 & 269 & \textbf{125.52} & 203  \\ \hline
pw05\_100.3   & 132.51  & 121  & 129.83        & 93  & 40.96 & 47 & \textbf{30.94}  & 33  \\ \hline
pw05\_100.4   & 793.85  & 777  & 585.08        & 411 & 222.35 & 263 & \textbf{133.66} & 179  \\ \hline
pw05\_100.5   & 190.30  & 175  & 123.76        & 85  & 51.56 & 73 & \textbf{33.14}  & 47   \\ \hline
pw05\_100.6   & 1401.88 & 1385 & 976.62        & 711 & 312.49 & 355 & \textbf{217.01} & 269  \\ \hline
pw05\_100.7   & 386.94  & 357  & 276.98        & 179 & 86.87 & 71 & \textbf{57.71} & 45  \\ \hline
pw05\_100.8   & 98.62   & 121  & 101.21        & 49  & 25.15 & 17 & \textbf{23.05}  & 13   \\ \hline
pw05\_100.9   & 323.26  & 307  & 321.72        & 327 & 113.98 & 129 & \textbf{65.91} & 97  \\  \toprule \bottomrule
pw09\_100.0  & 579.02  & 527  & 457.12        & 323 &  171.56 & 227& \textbf{107.15} & 163  \\ \hline
pw09\_100.1  & 821.02  & 753  & 706.71        & 543 & 248.79 & 311 & \textbf{153.76} & 231  \\ \hline
pw09\_100.2  & 292.42  & 275  & 248.28        & 175 & 88.34 & 109 & \textbf{46.98} & 127  \\ \hline
pw09\_100.3  & 153.66  & 153  & 170.37        & 113 & 52.44 & 39 & \textbf{34.71} & 27 \\ \hline
pw09\_100.4  & 394.53  & 373  & 310.72        & 213 & 109.10 & 147 & \textbf{73.68} & 93  \\ \hline
pw09\_100.5  & 489.02  & 519  & 513.91        & 321 & 295.02 & 145 & \textbf{129.74} & 111  \\ \hline
pw09\_100.6  & 390.50  & 337  & 278.82        & 241 & 113.95 & 131 & \textbf{72.65} & 77  \\ \hline
pw09\_100.7  & 1010.22 & 963  & 824.16        & 589 & 271.59 & 293 & \textbf{182.11} & 193  \\ \hline
pw09\_100.8  & 305.12  & 295  & 287.91        & 171 & 134.71 & 75 & \textbf{86.93} & 59  \\ \hline
pw09\_100.9  & 260.41  & 255  & 193.98        & 119 & 72.06 & 69 & \textbf{68.10} & 57  \\ \toprule
\end{tabular}
\caption{CPU times (s) and B\&B nodes to solve pw problems.}
\label{pw0}
\end{table}

\end{section}

\begin{section}{Parallel algorithm}
In this section, we describe how the algorithmic ingredients from the serial B\&B algorithm are combined into a parallel solver which utilizes the distributive memory parallelism. 
OpenMP (Open Multi-Processing) and MPI (Message Passing Interface) are two widely used paradigms available for parallel computing \cite{dagum1998openmp, Forum:1994:MMI:898758}. 
Algorithms based on B\&B can be parallelized on different levels. For example, shared memory parallelization using OpenMP can be applied to speed up the computation of bounds, and secondly, exploration of different branches of the B\&B tree can be done in parallel using the distributive scheme. In this case, several workers, each with its own memory, are evaluating B\&B nodes concurrently and using MPI to share important information via messages.
This is the approach we have taken. 

The load coordinator--worker paradigm with distributed work pools is applied, in which one process becomes the master process carefully managing the status of each worker, while the remaining processes concurrently explore the branches of the B\&B tree. Each worker has its own local priority queue of subproblems and the work is shared when one of them becomes idle. 
The master process keeps track of the status of each worker and acts as a load coordinator receiving messages and, based on their content, acts in the appropriate manner. In order to distinguish between different workers, each process is given a fundamental identifier, called rank of the worker. 

At the start of the algorithm, the master process reads and broadcasts the graph instance to the workers.
It is important that every process has the knowledge of the original graph, 
since construction of subproblems via branching or via received MPI messages is done based on this information. 
All the data about the B\&B nodes is encoded as an MPI structure, which is used in communication between different workers in order to efficiently exchange and construct the subproblems. 

Next, the load coordinator evaluates the root node and distributes the current best lower bound. After the bounding step, two new subproblems are generated, which are then sent to the first two idle processes. Afterwards, the load coordinator monitors the status of workers (idle or busy), counts the number of B\&B nodes, sends the ranks of free workers, and distributes the best solution found thus far.

After the initialization phase, the load coordinator waits for different types of messages sent by the workers. Firstly, if the worker's local queue is empty, the message  	
is sent informing the master that the process is idle and it can receive further work. Secondly, during the algorithm, working processes generate multiple candidates for optimal solutions. The master node keeps track of the current best optimal value and solution. When a new solution is received, the value is compared, updated if necessary, and distributed back during the communication phase. And thirdly, the master process helps the computing processes with work sharing by sending messages containing the ranks of idle workers. After a worker computes the lower and upper bounds, it compares these values to see if this branch of the B\&B tree can be safely pruned or if further branching is needed, and the construction of new subproblems takes place. 
In the latter case, a request message is sent to the load coordinator, asking for idle processes to share with one of the new generated subproblems and/or subproblems left in the queue from previous branching processes. 
If no idle worker is available, the generated subproblems remain in the worker's queue and the work continues locally. 
Otherwise, subproblems are encoded and sent to available idle workers. This is also where exchange of the current best lower bound happens.

When all the workers become idle (all local queues are empty), the master process sends a message to finish, and the algorithm terminates. 

Lastly, we elaborate on how efficient distribution of the subproblems to the workers is achieved and how the idle time is reduced. 
If the number of available workers is large, we need to reach a certain depth of the B\&B  tree in order for the processes to receive some work. Until this happens in the algorithm, workers are idle and resources are wasted. To fully  exploit all the  available resources,  we need a strategy for the worker processes  to start evaluating the nodes of the B\&B tree as soon as possible. This is again where the strategy using the variable $\diff$ from Section 5 benefits the algorithm. 
After the computation of the upper bound in the root node, the load coordinator distributes the value $\diff$ to the workers. When the first two idle processes evaluate the generated subproblems, the value of the basic SDP relaxation is typically such that the condition \eqref{diff} is not satisfied. Thus, on the first few levels of the B\&B tree, the workers compute only the basic SDP bound to faster evaluate the node, and idle processes quickly receive the generated subproblems.

\begin{subsection}{Numerical results - parallel algorithm}
We present the numerical results comparing the serial with the parallel version of the \texttt{MADAM} solver. As a Message Passing Interface library, we use Open MPI to exchange the data between the computing cores. Among the graphs that are used to test the sequential algorithm, we consider only those for which the solver needs more than 200 seconds to compute the optimal solution. We solve the same problem with increasing number of cores and compare the times against the serial algorithm. The results are reported in Table \ref{parallel_results}.

For the hardest instance w09\_100.1 for the serial solver, we also present the scalability plot to demonstrate the efficiency of the parallel version. Due to the increased number of workers, the computational time drops. However, the larger number of workers also increases communication costs since multiple request are sent to the load coordinator. This is shown in Figure \ref{scalability} as the curve diverges from the ideal scaling curve.  


\begin{table}[!htbp]
\footnotesize
\centering
\begin{tabular}{|c|c| r r r r r|}
\bottomrule
&
&
  \multicolumn{5}{c|}{\texttt{MADAM} - parallel}\\ \cline{3-7} 
   &
   &
   \multicolumn{5}{c|}{Number of workers}\\ \cline{3-7} 

    \multirow{-3}{*}{Problem} &
  \multirow{-3}{*}{\texttt{MADAM} - serial} &
  \multicolumn{1}{c}{3} &
  \multicolumn{1}{c}{6} &
  \multicolumn{1}{c}{12} &
  \multicolumn{1}{c}{24} &
  \multicolumn{1}{c|}{48} 
  \\ \toprule \bottomrule
  g05\_100.1  &  494.10   &  206.37  &     109.97     & 68.53   &  37.25  &  \textbf{23.85}   \\ \hline
  pm1d\_100.1 &  267.08  & 100.23   &    53.51      &   30.98 & 17.53  &   \textbf{10.92}  \\ \hline
w09\_100.1 &  635.93  &  226.02  &    125.19      &  72.43  &  41.37  &  \textbf{26.46} \\ \hline
w09\_100.2 &   222.66  &  75.64   &     43.17     &  26.80  &  14.64  &  \textbf{11.02}  \\ \hline
w09\_100.3 &  263.81  &  106.76  &      57.96    &  34.42  &  19.75  &   \textbf{15.17}  \\ \hline
pw05\_100.0 &  274.04  & 95.95   &    50.59      & 29.16  &  17.09 &  \textbf{10.71}  \\ \hline
pw05\_100.6 &  217.01   &  84.11  &    51.39     &  26.36  &  18.97  &  \textbf{15.46}  \\ \toprule
\end{tabular}
\caption{Numerical results obtained with parallel version of \texttt{MADAM}.}
\label{parallel_results}
\end{table}

\begin{figure}[!htbp]
\centering
\begin{subfigure}{.5\textwidth}
  \centering
  \includegraphics[width=1\linewidth]{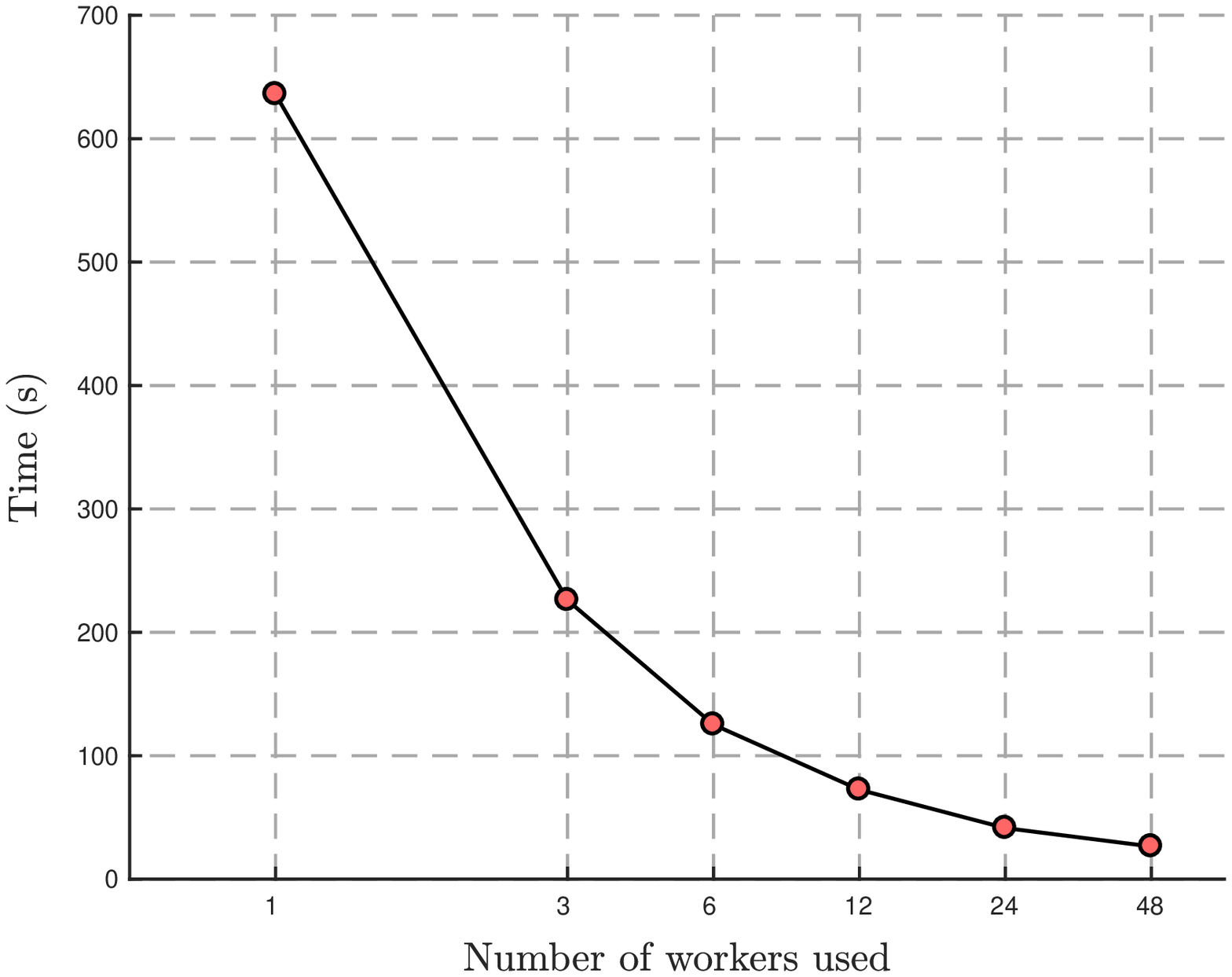}
\end{subfigure}%
\begin{subfigure}{.5\textwidth}
  \centering
  \includegraphics[width=1\linewidth]{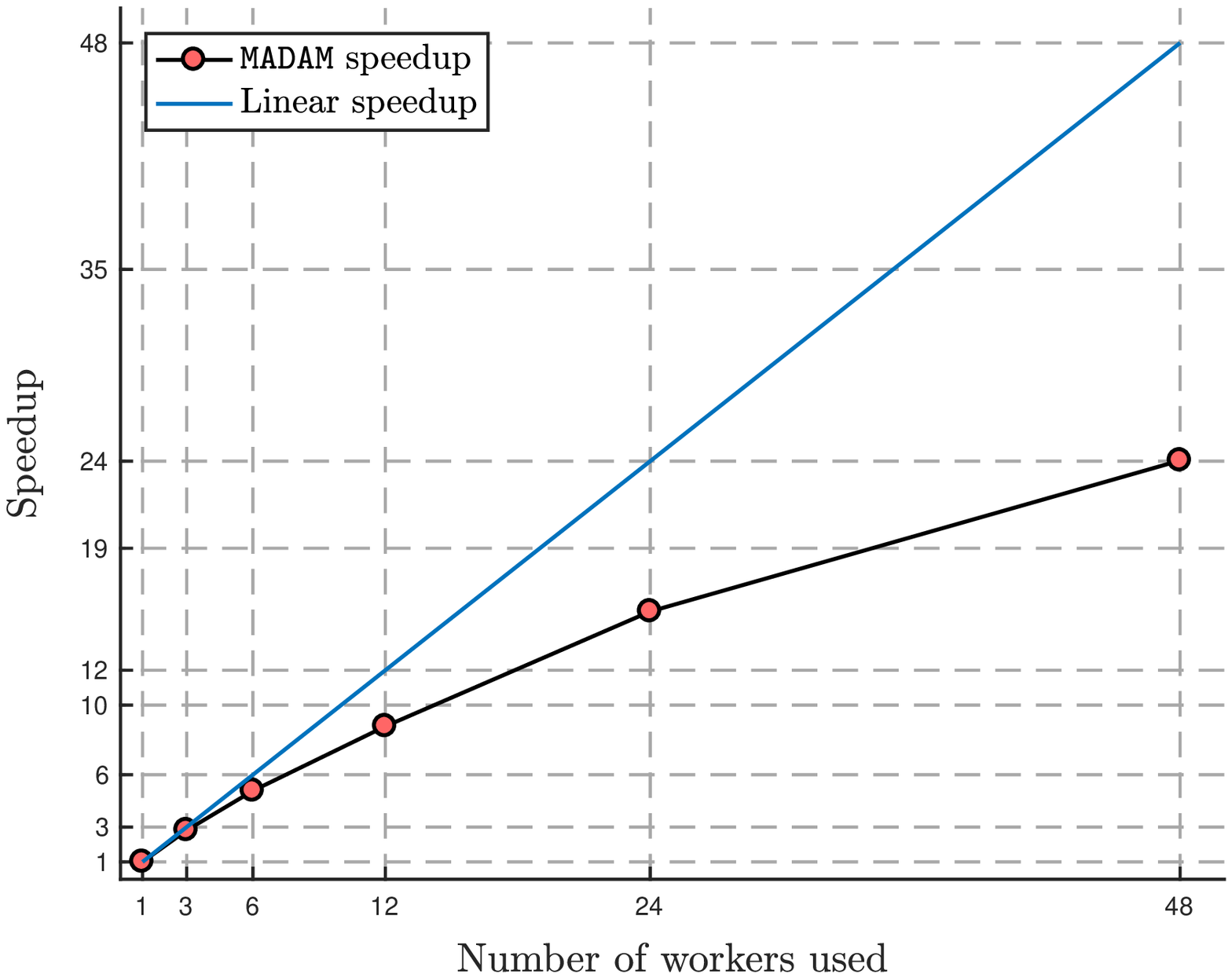}
\end{subfigure}
\caption{The plots depict scalability of \texttt{MADAM} solver for the w09\_100.1 problem. Due to increased number of workers, the computational
	time drops. However, larger number of workers also increases communication costs, which results in the divergence of the curve 
	from the ideal scaling curve (right).}
\label{scalability}
\end{figure}

Next, we use \verb|rudy| \cite{rinaldirudy}, a machine independent graph generator,  to construct new Max-Cut instances with 180 vertices, similar to the ones in the BiqMac library. The parameters for generating the graphs can be found in \cite{wiegele2006nonlinear} in Appendix B. 
Table \ref{madam_180} presents numerical results  obtained by applying the parallel solver using 240 cores for all instances. We report the maximum cut values and the number of B\&B nodes needed to solve the problems to optimality. 

\begin{table}[!htbp]
\begin{minipage}{.5\linewidth}
\footnotesize
\centering
\begin{tabular}{|c|c|r r|}
\bottomrule
&
  \multicolumn{3}{c|}{\texttt{MADAM} - parallel} \\ \cline{2-4} 
  \multirow{-2}{*}{Problem} &
  Solution &
  Time &
  Nodes \\ \toprule \bottomrule
g05\_180.0     &       4507       &  \textbf{671.23} &  148617 \\ \hline
g05\_180.1       &       4500       &   \textbf{670.63} & 137665 \\ \hline
g05\_180.2       & 4503               &  \textbf{1116.24}  & 281215  \\ \hline
g05\_180.3        & 4495               & \textbf{3706.61} & 786457  \\ \hline
g05\_180.4         & 4492         & \textbf{5209.59}  & 1556485  \\ \hline
g05\_180.5          &4488             & \textbf{8964.00}  & 2333997  \\ \hline
g05\_180.6          & 4489            & \textbf{10542.41}  & 2298681  \\ \hline
g05\_180.7        &4490             & \textbf{8952.15} & 2168339  \\ \hline
g05\_180.8         & 4490            & \textbf{7431.42}  & 2413043  \\ \hline
g05\_180.9           &4492           & \textbf{4092.88}  & 1083549  \\ \toprule
\end{tabular}
\end{minipage}%
\begin{minipage}{.5\linewidth}
\footnotesize
\centering
\begin{tabular}{|c|c|r r|}
\bottomrule
&
  \multicolumn{3}{c|}{\texttt{MADAM} - parallel} \\ \cline{2-4} 
  \multirow{-2}{*}{Problem} &
  Solution &
  Time &
  Nodes \\ \toprule \bottomrule
pm1d\_180.0          & 884         &  \textbf{9907.77} & 1227367  \\ \hline
pm1d\_180.1            & 930       &   \textbf{1098.64} & 164965 \\ \hline
pm1d\_180.2            & 865          &  \textbf{374.01}  & 54539  \\ \hline
pm1d\_180.3            &866           & \textbf{442.40} &  73171  \\ \hline
pm1d\_180.4       & 843         & \textbf{1567.66}  & 215171   \\ \hline
pm1d\_180.5             & 785         & \textbf{7697.99}  & 1916281  \\ \hline
pm1d\_180.6                & 747      & \textbf{963.60}  & 206105  \\ \hline
pm1d\_180.7          &  933        & \textbf{631.29} & 49671  \\ \hline
pm1d\_180.8        &  902          & \textbf{5146.23}  & 958197 \\ \hline
pm1d\_180.9        & 950             & \textbf{3892.91}  & 1377653  \\ \toprule
\end{tabular}
\end{minipage}%
\end{table}

\begin{table}[!htbp]
\begin{minipage}{.5\linewidth}
\footnotesize
\centering
\begin{tabular}{|c|c|r r|}
\bottomrule
&
  \multicolumn{3}{c|}{\texttt{MADAM} - parallel} \\ \cline{2-4} 
  \multirow{-2}{*}{Problem} &
  Solution &
  Time &
  Nodes \\ \toprule \bottomrule
  w05\_180.0     &  4126            &  \textbf{4615.81} &  243631 \\ \hline
w05\_180.1       &  4160           &   \textbf{19464.53} & 2826185 \\ \hline
w05\_180.2       &  3697             &  \textbf{2708.77}  & 377159   \\ \hline
w05\_180.3        &   3718            & \textbf{1894.69} & 235077  \\ \hline
w05\_180.4         &  3777       & \textbf{6718.74}  &830525 \\ \hline
w05\_180.5          &  3883          & \textbf{4640.78}  & 534895  \\ \hline
w05\_180.6          & 4233           & \textbf{10021.08}  & 1806407 \\ \hline
w05\_180.7        & 3830           & \textbf{9354.87} & 1178125  \\ \hline
w05\_180.8         &   3801          & \textbf{4440.65}  & 550265  \\ \hline
w05\_180.9           &   3906        & \textbf{3715.36}  & 576693  \\ \toprule
\end{tabular}
\end{minipage}%
\begin{minipage}{.5\linewidth}
\footnotesize
\centering
\begin{tabular}{|c|c|r r|}
\bottomrule
&
  \multicolumn{3}{c|}{\texttt{MADAM} - parallel} \\ \cline{2-4} 
  \multirow{-2}{*}{Problem} &
  Solution &
  Time &
  Nodes \\ \toprule \bottomrule
w09\_180.0          &   5550       &  \textbf{13585.11} & 1204445  \\ \hline
w09\_180.1            &  5397      &   \textbf{8165.05} & 905925\\ \hline
w09\_180.2            & 4710          &  \textbf{6399.73}  & 613239  \\ \hline
w09\_180.3            &   5003        & \textbf{16577.92} &  1781313 \\ \hline
w09\_180.4       &  5267       & \textbf{8915.48}  & 1920259  \\ \hline
w09\_180.5             & 4977          & \textbf{3402.87}  & 237069  \\ \hline
w09\_180.6                & 5744      & \textbf{7312.46}  & 954457 \\ \hline
w09\_180.7          & 5388       & \textbf{33339.01} & 4722991  \\ \hline
w09\_180.8        & 4907       & \textbf{52024.30}  & 12930415  \\ \hline
w09\_180.9        &  4735            & \textbf{14737.37}  & 1777259  \\ \toprule
\end{tabular}
\end{minipage}%
\end{table}

\begin{table}[!htbp]
\begin{minipage}{.5\linewidth}
\footnotesize
\centering
\begin{tabular}{|c|c|r r|}
\bottomrule
&
  \multicolumn{3}{c|}{\texttt{MADAM} - parallel} \\ \cline{2-4} 
  \multirow{-2}{*}{Problem} &
  Solution &
  Time &
  Nodes \\ \toprule \bottomrule
pw05\_180.0     &    25536         &  \textbf{22309.20} &  3324025 \\ \hline
pw05\_180.1       &   25301          &   \textbf{1111.58} & 146393 \\ \hline
pw05\_180.2       &   25260          &  \textbf{6354.55}  & 1036951   \\ \hline
pw05\_180.3        &  25514               & \textbf{411.39} & 30617  \\ \hline
pw05\_180.4         &   25128       & \textbf{9534.08}  & 1185213 \\ \hline
pw05\_180.5          &  25338          & \textbf{4751.92}  & 481553  \\ \hline
pw05\_180.6          &  25200        & \textbf{20387.66}  & 3626241 \\ \hline
pw05\_180.7        &  25373         & \textbf{707.48} & 131639  \\ \hline
pw05\_180.8         &  25215           & \textbf{3118.33}  & 395689  \\ \hline
pw05\_180.9           &  25608         & \textbf{20099.60}  & 2309945  \\ \toprule
\end{tabular}
\end{minipage}%
\begin{minipage}{.5\linewidth}
\footnotesize
\centering
\begin{tabular}{|c|c|r r|}
\bottomrule
&
  \multicolumn{3}{c|}{\texttt{MADAM} - parallel} \\ \cline{2-4} 
  \multirow{-2}{*}{Problem} &
  Solution &
  Time &
  Nodes \\ \toprule \bottomrule
pw09\_180.0          &   43070       &  \textbf{9750.36} & 2214985	  \\ \hline
pw09\_180.1            &   42820     &   \textbf{2441.80} & 462807 \\ \hline
pw09\_180.2            &   42736        &  \textbf{2820.30}  & 195949  \\ \hline
pw09\_180.3            &   42880         & \textbf{3489.30} &  473039 \\ \hline
pw09\_180.4       &  42784     & \textbf{1619.40}  & 197847  \\ \hline
pw09\_180.5             & 42969         & \textbf{9087.99}  & 748497  \\ \hline
pw09\_180.6                & 42934      & \textbf{5775.62}  &729985 \\ \hline
pw09\_180.7          &   43210      & \textbf{13250.48} & 2430053  \\ \hline
pw09\_180.8        &   42806        & \textbf{26826.61}  &  4179147 \\ \hline
pw09\_180.9        &   43438           & \textbf{322.01}  & 39889  \\ \toprule
\end{tabular}
\end{minipage}%
\caption{Numerical results obtained with the parallel version of \texttt{MADAM} for graphs with 180 vertices.}
\label{madam_180}
\end{table}

\end{subsection}

\end{section}

\begin{section}{Conclusions and future work}
In this paper, we have presented an efficient exact solver \texttt{MADAM} for the Max-Cut problem that applies the alternating direction method of multipliers to efficiently compute high-quality SDP-based upper bounds. 
This is due to the small computational complexity per iteration of the proposed method, since it essentially consists of solving one sparse system of linear equations and projection onto the nonnegative orthant and positive semidefinite cone. Furthermore, tight upper bounds are obtained due to the inclusion of a subset of hypermetric inequalities.  Numerical results show that the {\texttt{MADAM}} solver outperforms state-of-the-art approaches. 
By combining the elements of the serial algorithm and the new strategy for a faster exploration of the B\&B tree, we have developed a parallel solver based on MPI that utilizes the load coordinator-worker paradigm. We have shown that the {\texttt{MADAM}} solver scales very well and that we can greatly reduce the time needed to solve the Max-Cut problem to optimality and increase the size of instances that can be solved in a routine way. 

In the future, we intend to increase the performance of the algorithm by using one-sided MPI communication, thus enabling a complete removal of the master process from the algorithm and a more efficient exchange of the messages.  
We will also further improve the performance of \texttt{MADAM} by passing cutting planes and the output of the bounding routine from parent to children nodes. We believe that this will help accelerate the ADMM method and the search for prominent inequalities.
\end{section}

\section*{Acknowledgements}
The authors would like to thank the anonymous referees for their careful reading of the paper and for their constructive comments, which are greatly appreciated. This work was supported by the Slovenian Research Agency: projects N1-0057, N1-0071, J1-1691, J1-8130,  J1-8132. 
The first author also acknowledges the financial support from the Slovenian Research Agency under the program for young researchers (MR+).

\footnotesize{
\bibliography{reference} 
\bibliographystyle{plain}
}

\end{document}